\documentclass[preprint]{imsart}
\usepackage{amsmath}
\usepackage{amsthm}
\usepackage{amssymb}
\usepackage{footmisc}
\usepackage[authoryear,round]{natbib}
\usepackage[colorlinks=true,linkcolor=blue,citecolor=blue]{hyperref}
\usepackage{hypernat}
\usepackage{verbatim}
\usepackage{textcomp}
\usepackage{enumerate}
\usepackage{graphicx}
\usepackage{array}
\newcolumntype{H}{>{\setbox0=\hbox\bgroup}c<{\egroup}@{}}

\usepackage{fancyhdr}
 
\newcommand{\R}{{\mathbb R}}
\newcommand{\E}{{\mathbb E}}
\renewcommand{\P}{{\mathbb P}}
\newcommand{\N}{{\mathbb N}}

\newcommand{\eps}{\varepsilon}

\newcommand{\argmin}[1]{{\operatorname{argmin}}_{#1}}

\newcommand{\lmin}[1]{{\operatorname{\lambda_{\text{min}}}}{#1}}

\DeclareMathOperator{\Var}{Var}
\DeclareMathOperator{\Cov}{Cov}
\DeclareMathOperator{\trace}{trace}
\DeclareMathOperator{\sign}{sign}
\DeclareMathOperator{\diag}{diag}

\newtheorem{theorem}{Theorem}[section]
\newtheorem{proposition}[theorem]{Proposition}
\newtheorem{lemma}[theorem]{Lemma}

\newtheorem{remark}[theorem]{Remark}
\newtheorem*{remark*}{Remark}

\newtheorem*{definition*}{Definition}

\numberwithin{equation}{section}

\newcounter{rcnt}[section]

\newcommand{\rem}[1]{}

\newcounter{desccount}

\newcommand{\descref}[1]{\hyperref[#1]{#1}}

\begin{document}


\title{Leave-one-out prediction intervals in linear regression models with many variables}
\runauthor{Steinberger, Leeb}

\runtitle{Prediction intervals in linear regression}

\begin{aug}
\author{\fnms{Lukas} \snm{Steinberger}}
\and
\author{\fnms{Hannes} \snm{Leeb}}

\affiliation{University of Vienna}

\address{
	Department of Statistics and Operations Research\\
	University of Vienna  \\
	Oskar-Morgenstern-Platz 1 \\
	1090 Vienna, Austria
	}
	
\end{aug}

\begin{abstract}
We study prediction intervals based on leave-one-out residuals in a linear regression model where the number of explanatory variables can be large compared to sample size. We establish uniform asymptotic validity (conditional on the training sample) of the proposed interval under minimal assumptions on the unknown error distribution and the high dimensional design. Our intervals are generic in the sense that they are valid for a large class of linear predictors used to obtain a point forecast, such as robust $M$-estimators, James-Stein type estimators and penalized estimators like the LASSO. 
These results show that despite the serious problems of resampling procedures for inference on the unknown parameters \citep[cf.][]{Bickel83, Mammen96, ElKaroui15}, leave-one-out methods can be successfully applied to obtain reliable predictive inference even in high dimensions. 
\end{abstract}

\maketitle


\section{Introduction}

When a prediction for the value of a response variable is calculated, based on the corresponding explanatory variables and a training sample, then a prediction interval provides important additional information about the uncertainty associated with the point prediction. 
We suggest a very simple method based on leave-one-out residuals which is generic in the sense that it provides asymptotically honest prediction intervals for a large class of possible point predictors. 
Our approach is similar, in spirit, to the methods proposed in \citet{ButlerRoth80} and \citet{Stine85} \citep[see also][]{Schmoyer92}, in the sense that we rely on resampling and leave-one-out ideas for predictive inference. However, the methods from these references, like most resampling procedures in the literature, are investigated only in the classical large sample asymptotic regime. Notable exceptions are \citet{Bickel83}, \citet{Mammen96} and, recently, \citet{ElKaroui15}. These articles draw mainly negative conclusions, arguing, for instance, that the famous residual bootstrap in linear regression, which relies on the consistent estimation of the true unknown error distribution, is unreliable when the number of variables in the model is not small compared to sample size. In contrast, our method does not suffer from these problems because we directly estimate the conditional distribution of the prediction error instead of the true unknown distribution of the disturbances.

Our work is greatly inspired by \cite{ElKaroui13b} and \citet{Bean13} \citep[see also][]{ElKaroui13}, who investigate the efficiency of general $M$-estimators in linear regression when the number of regressors $p$ is of the same order of magnitude as sample size $n$. In particular, the $M$-estimators studied in these references provide one leading example of a class of linear predictors that are compatible with our construction of prediction intervals. Other classes of predictors that can be used with our method are those based on James-Stein type shrinkage estimators as well as penalized estimators like the LASSO (see Section~\ref{sec:estimators}).

\subsection{A generic prediction interval in linear regression}

Given i.i.d. training data $(y_i, x_i)_{i=1}^n$ and a feature vector $x_0$, our goal is to predict $y_0$, where all pairs obey the linear model
$$
y_i \;=\; \beta'x_i \;+\; u_i \hspace{1cm} i=0,1,\dots, n.
$$ 
Here, $\beta=\beta_n\in\R^{p_n}$, $x_i$ is independent of the error $u_i$, with $\E[x_1]=0$ and $\E[x_1x_1']=\Sigma=\Sigma_n\in S_{p_n}$, where $S_{p_n}$ is the set of symmetric, positive definite $p_n\times p_n$ matrices. In this model linear predictors of the form $x_0'\hat{\beta}$ are the method of choice, where $\hat{\beta} = \hat{\beta}(Y,X)$ is an estimator for $\beta$ based on the training data $Y= (y_1,\dots, y_n)'$, $X=[x_1,\dots, x_n]'$. Moreover, we want to provide a measure of uncertainty associated with this prediction. In particular, we want to construct a conditionally asymptotically honest $(1-\alpha)$ prediction interval for the future response value $y_0$, i.e., an interval valued function $PI_\alpha (Y, X, x_0)$ such that
\begin{align}\label{eq:honest}
\sup_{\substack{\beta\in\R^{p_n}, \sigma^2>0\\ \Sigma\in S_{p_n}}} \E_{\beta,\sigma^2, \Sigma}\left[ 
	\left| \P_{\beta,\sigma^2, \Sigma} \left( y_0 \in PI_\alpha(Y,X,x_0) \Big| Y,X \right) - (1-\alpha) \right| \right] \; \xrightarrow[n\to\infty]{} \; 0.
\end{align}
This property should remain to hold even if $p_n$ is allowed to grow with $n$ such that $p_n/n\gg0$. We also want our prediction interval to be generic, so that it can be applied irrespective of the estimator $\hat{\beta}$ used to obtain a point prediction. We propose the following simple construction, which, to the best of our knowledge, has never been studied in the present asymptotic framework of $p_n/n\gg0$.

When the estimator $\hat{\beta}$ is used to obtain a prediction $x_0'\hat{\beta}$ for $y_0$, we construct the leave-one-out prediction interval $PI_\alpha^{(L1O)}$ as follows. For $\alpha\in(0,1)$, let $\tilde{q}_{n,\alpha}$ denote an empirical $\alpha$ quantile of the sample $\tilde{u}_1,\dots, \tilde{u}_n$ of leave-one-out residuals $\tilde{u}_i = y_i - x_i'\hat{\beta}_{(i)}$, where $\hat{\beta}_{(i)}$ is the estimator calculated from the sample of size $n-1$ obtained by removing the observation $(y_i,x_i)$. Then $PI_\alpha^{(L1O)}$ is given by
\begin{align}\label{eq:L1OPI}
PI_\alpha^{(L1O)}(Y,X,x_0) \;=\; \Big[ x_0'\hat{\beta} \;+\; \tilde{q}_{n,\alpha/2}, \;x_0'\hat{\beta} \;+\;\tilde{q}_{n,1-\alpha/2}\Big].
\end{align}
Theorem~\ref{thm:L1OPI} below shows that this prediction interval has the property~\eqref{eq:honest}, provided that the sequence of estimators $\hat{\beta}_n$ used to obtain a point prediction is sufficiently regular (see Section~\ref{sec:main}).

The idea behind this procedure is remarkably simple. When predicting $y_0$ by $x_0'\hat{\beta}$, we make the error $y_0 - x_0'\hat{\beta}$, which we do not observe in practice. To estimate the distribution of this prediction error we simply use the empirical distribution of the leave-one-out residuals $\tilde{u}_i = y_i -x_i'\hat{\beta}_{(i)}$. Notice that $\hat{\beta} = \hat{\beta}(Y,X)$ is independent of $(y_0, x_0)$, and $\hat{\beta}_{(i)}$ is independent of $(y_i, x_i)$, and thus, $\tilde{u}_i$ has almost the same distribution as the prediction error, except that $\hat{\beta}_{(i)}$ is calculated from one observation less than $\hat{\beta}$. In most cases this difference turns out to be negligible if $n$ is large, even if $p_n$ is large too. Once we have estimated the distribution of the prediction error $y_0 - x_0'\hat{\beta}$ it is straight forward to construct a prediction interval for $y_0$ using quantiles. 
Moreover, it turns out that in the same way one can estimate the conditional distribution of $y_0 - x_0'\hat{\beta}$, given the training sample $Y, X$, and that this estimation is successful even in a situation where there are many variables in the regression and where the feature vectors $x_i$ have a complex geometric structure, in the sense that possibly $| \|x_1/\sqrt{p_n}\|_2 - 1| \gg 0$ \citep[See, e.g.,][Section 3.2, for a discussion why this is desirable.]{ElKaroui10}.

The remainder of the paper is organized as follows. In Section~\ref{sec:main} we state our main result Theorem~\ref{thm:L1OPI} together with a detailed description of the technical conditions needed for its proof. In Section~\ref{sec:estimators} we specifically discuss the conditions imposed on the estimators $\hat{\beta}$ used to construct the point prediction $x_0'\hat{\beta}$, and we provide some leading examples of estimators which satisfy these conditions. Section~\ref{sec:conclusion} concludes.

\section{Main results}
\label{sec:main}

To define the data generating process, consider a double infinite array $V_0 = \{v_{ij} : i,j=0,1,2,\dots\}$ of i.i.d. random variables with mean zero, unit variance and finite fourth moment, and two sequences $L_0 = \{l_0, l_1, \dots\}$ and $U_0 =\{u_0,u_1, \dots\}$, each consisting of i.i.d. random variables where $l_0$ satisfies $\E[l_0^2] = 1$ and $|l_0|\ge c > 0$. The arrays $V_0$, $L_0$ and $U_0$ are assumed to be jointly independent, but no assumptions are imposed on the distribution of $u_0$. Throughout, we consider the case where the number $p=p_n$ of explanatory variables in the model depends on sample size $n$. Moreover, we consider the class $S_{p_n}$ of $p_n\times p_n$ positive definite covariance matrices and let the unknown covariance of the design be denoted by $\Sigma_n\in S_{p_n}$. By $\Sigma^{1/2}$ we always refer to the unique symmetric and positive definite square root of $\Sigma$. The training data are given by the i.i.d. pairs $(y_{i,n},x_{i,n})_{i=1}^n$, the observed regressor vector in the forecast period is given by $x_{0,n}$ and the corresponding unobserved response variable is $y_{0,n}$, where $x_{i,n} = \Sigma_n^{1/2}l_i (v_{i1},\dots, v_{ip_n})'$,  
\begin{align}
y_{i,n} \;=\; \beta_n'x_{i,n} \;+\; \sigma_n u_i, \hspace{1cm} \text{for } i=0,1,2,\dots,
\end{align} 
and where $\beta_n \in\R^{p_n}$ and $\sigma_n>0$ are unknown parameters. We adopt the usual matrix notation $Y = (y_{1,n},\dots, y_{n,n})'$ and $X = [x_{1,n},\dots, x_{n,n}]'$. Most of the time we will drop the dependence on $n$ in the notation for convenience.

We also need to specify the class of estimators for which the leave-one-out prediction interval provides valid inference. The following set of conditions is used in Theorem~\ref{thm:L1OPI} below.

\begin{enumerate}
\renewcommand{\theenumi}{(C\arabic{enumi})}
\renewcommand{\labelenumi}{\textbf{\theenumi}}

\item \label{c.gauss} 
\begin{enumerate}
\renewcommand{\theenumii}{\alph{enumii}}
\renewcommand\labelenumii{(\theenumii)}
\makeatletter \renewcommand\p@enumii{} \makeatother

	\item \label{c.gauss1}
		For every $n$, the estimator $\hat{\beta}_n : \R^{n\times (p_n+1)} \mapsto \R^{p_n}$, is symmetric in the observations $w_i = (y_i, x_i')'$ in the sense that 
			$$
				\hat{\beta}_n\left( [w_1,\dots, w_n]' \right) = \hat{\beta}_n\left( [w_{\pi(1)},\dots, w_{\pi(n)}]' \right),
			$$
			for every choice of $w_i\in\R^{p_n+1}$ and for every permutation $\pi$ of $n$ elements.
			
	\item \label{c.gauss2}
		There exists $\tau\in[0,\infty)$, such that for every $\eps >0$,
		$$
		\sup_{\substack{\beta\in\R^{p_n}, \sigma^2>0\\ \Sigma\in S_{p_n}}} \P_{\beta, \sigma^2, \Sigma}
			\left(\left| \|\Sigma^{1/2}(\hat{\beta}_n -\beta)/\sigma\|_2 \; - \; \tau\right| > \eps \right)
			\xrightarrow[n\to\infty]{} 0.
		$$
		
	\item \label{c.gauss3}
		For every $\eps>0$,
		$$
		\sup_{\substack{\beta\in\R^{p_n}, \sigma^2>0\\ \Sigma\in S_{p_n}}} \P_{\beta, \sigma^2, \Sigma}
			\left( \left\|
				\Sigma^{1/2}(\hat{\beta}_n - \hat{\beta}_{(1),n})/\sigma
			\right\|_2 > \eps \right)
			\xrightarrow[n\to\infty]{} 0.
		$$		
	\end{enumerate}

\item \label{c.general} 
		There exists $\delta\in(0,2]$, such that for every $\eps>0$, 
		$$ 
		\sup_{\substack{\beta\in\R^{p_n}, \sigma^2>0\\ \Sigma\in S_{p_n}}} \P_{\beta, \sigma^2, \Sigma}
			\left( \|\Sigma^{1/2}(\hat{\beta}_n -\beta)/\sigma\|_{2+\delta} > \eps \right)
			\xrightarrow[n\to\infty]{} 0.
		$$
\end{enumerate}

Under these specifications we can proof our main result on the leave-one-out prediction interval in \eqref{eq:L1OPI}. 

\begin{theorem}\label{thm:L1OPI}
In the data generating model introduced above the following holds true:
\begin{enumerate}
	\setlength\leftmargin{-20pt}
	\renewcommand{\theenumi}{(\roman{enumi})}
	\renewcommand{\labelenumi}{{\theenumi}} 
\item \label{thm:tau>0}
If Condition~\ref{c.gauss} holds with $\tau>0$ and $p_n\to\infty$ as $n\to\infty$, then either Condition~\ref{c.general} or the requirement that $v_{ij} \thicksim \mathcal N(0,1)$ imply that the leave-one-out prediction interval $PI_\alpha^{(L1O)}$ in \eqref{eq:L1OPI} satisfies \eqref{eq:honest}.

\item \label{thm:tau=0}
If Condition~\ref{c.gauss} holds with $\tau = 0$ and the distribution of the error $u_0$ has a continuous and strictly increasing distribution function, then the leave-one-out prediction interval $PI_\alpha^{(L1O)}$ in \eqref{eq:L1OPI} satisfies \eqref{eq:honest}.

\item \label{thm:adaptation}
If the assumptions of either case \ref{thm:tau>0} or \ref{thm:tau=0} above are satisfied, then the scaled length $\ell_n := |PI_\alpha^{(L1O)}|/\sigma_n = (\tilde{q}_{n,1-\alpha/2} - \tilde{q}_{n,\alpha/2})/\sigma_n$ of the leave-one-out prediction interval $PI_\alpha^{(L1O)}$ in \eqref{eq:L1OPI} satisfies $\ell_n \to \ell_\alpha(\tau)$ in probability as $n\to\infty$, where $\ell_\alpha(\tau) = q_{1-\alpha/2} - q_{\alpha/2}$ is the corresponding inter-quantile range of the distribution of $l_0N\tau + u_0$, and where $N\thicksim \mathcal N(0,1)$ is independent of $(l_0,u_0)$.

\end{enumerate}
\end{theorem}

Theorem~\ref{thm:L1OPI} establishes the conditional asymptotic validity of $PI_\alpha^{(L1O)}$ under virtually no assumptions on the true error distribution and without any restriction of the relative growth of $p_n$ and $n$. Of course, this is possible only because of the high level assumptions \ref{c.gauss} and \ref{c.general} on the estimator sequences. To verify these conditions for a specific sequence of $\hat{\beta}_n$, somewhat stronger assumptions on the data generating process than those imposed by the theorem are typically needed. In Section~\ref{sec:estimators} we discuss several classes of estimators that exhibit these required properties under different sets of assumptions on the data generating process. In particular, we will see that when the number of parameters $p_n$ increases at the same rate as $n$ such that $p_n/n\to\kappa>0$, then many classical estimators will have a scaled conditional mean squared prediction error $\E[|x_0'\beta - x_0'\hat{\beta}_n|^2/\sigma^2|Y,X] = \|\Sigma^{1/2}(\hat{\beta}_n-\beta)/\sigma\|_2^2$ that converges to a non-random limit $\tau^2$ which is strictly positive (as required in Theorem~\ref{thm:L1OPI}.\ref{thm:tau>0}). In this scenario, the performance of different predictors can be distinguished by their respective value of $\tau$, which, however, is not observed in practice \citep[cf.][]{ElKaroui13, Bean13}. Alternatively, the length of a prediction interval based on a point predictor can also serve as a performance measure for the predictor.
Part~\ref{thm:adaptation} of the theorem reveals how the length of $PI_\alpha^{(L1O)}$ depends on $\tau$ and the distribution of $(l_0,u_0)$. If the error distribution is sufficiently regular, then the asymptotic length $\ell_\alpha(\tau)$ is monotonic in $\tau$ (cf. Remark~\ref{rem:adaptation}). If $p_n/n\to0$, most reasonable estimators will be consistent and thus lead to a value of $\tau=0$, under sufficient regularity conditions. In this scenario, the asymptotic length of $PI_\alpha^{(L1O)}$ depends only on the distribution of the error term $u_0$. 
In the case \ref{thm:tau=0} $\tau=0$, the additional assumption on the error distribution is needed to guarantee that the quantiles of the asymptotic distribution of the prediction error are unique also in this case. Without this assumption it can be shown that the asymptotic coverage probability in \eqref{eq:honest} is no less than $1-\alpha$. 

The assumptions imposed on the design distribution are relatively mild. In particular, the entries of the standardized design vectors $\Sigma^{-1/2}x_i$ are not required to be independent because the scalar random variable $l_i$ can introduce some form of dependence. The $l_i$ also have the effect that $| \|\Sigma^{-1/2}x_i\|_2/\sqrt{p_n} - 1| \gg 0$, even if $p_n$ is large. This feature is not present in most designs studied in the literature on high-dimensional regression, where it is often implicitly assumed that all the design vectors $x_i$ are closely concentrated on a sphere of radius $\sqrt{p_n}$ \citep[cf. the discussion in][Section 3.2]{ElKaroui10}. The assumptions on the array $V_0$ are standard in random matrix theory. The existence of the fourth moment is needed in the proof of Theorem~\ref{thm:L1OPI}\ref{thm:tau>0} in case that \ref{c.general} holds only with $\delta=2$ and could be relaxed to a $2+\delta$ moment restriction. However, the fourth moment is also essential in establishing the properties \ref{c.gauss} and \ref{c.general} for the ordinary least squares estimator (cf. Proposition~\ref{prop:OLS} below).

Also note that in part~\ref{thm:tau>0} we have not imposed any conditions on the distribution of the error term $u_0$. A minimal set of assumptions on this distribution is usually needed in order to prove the regularity \ref{c.gauss} or \ref{c.general} of a given estimator $\hat{\beta}$ used for prediction. But such assumptions may be very mild, not even requiring existence of any moments \citep[cf.][for LAD estimation in case $p_n=p$ is fixed]{Zhao93}, and they are not needed at all for the proof of Theorem~\ref{thm:L1OPI} once the regularity \ref{c.gauss} and/or \ref{c.general} is given.

\begin{remark}\normalfont
We note that the leave-one-out prediction interval $PI_\alpha^{(L1O)}$ is, in general, not symmetric about the point prediction $x_0'\hat{\beta}$, even if the $\alpha/2$ and $1-\alpha/2$ quantiles are used, due to a possibly skewed error distribution. Moreover, since the procedure is based on empirical quantiles, it immediately allows for the construction of one-sided intervals.
\end{remark}

\begin{remark}\normalfont\label{rem:adaptation}
As discussed above, predictor accuracy is measured by the asymptotic conditional mean squared prediction error $\tau^2$ (cf. \ref{c.gauss}.(\ref{c.gauss2})). Intuitively, a more accurate point predictor should also lead to a shorter prediction interval, so that interval length is equivalent to mean squared prediction error as a relative performance measure. However, Theorem~\ref{thm:L1OPI}.\ref{thm:adaptation} shows that the dependence of the asymptotic interval length $\ell_\alpha(\tau)$ on $\tau$ is mediated by the error distribution in such a way that $\tau\mapsto \ell_\alpha(\tau)$ may be non-monotonic. Consider, for example, the case where $l_0=1$ and $u_0$ takes only the values $-1$ and $1$ with equal probability. 
A sufficient condition for $\tau\mapsto \ell_\alpha(\tau)$ to be non-decreasing on $(0,\infty)$ for all $\alpha\in(0,1)$ is the dispersiveness of the error distribution in the sense of \citet[][Section 5]{Lewis81}. To see this, simply note that $l_0 N \tau_2$ is more dispersed than $l_0 N \tau_1$ if $\tau_1\le \tau_2$, because the quantile function of $l_0 N \tau$ is given by $\tau$ times the quantile function of $l_0 N$. The class of absolutely continuous dispersive distributions is characterized by log-concave densities \citep[][Theorem 8]{Lewis81}. Thus, if $u_0$ has a log-concave density, then $\ell_\alpha(\tau)$ and $\tau$ are compatible measures of the relative performance of two point predictors and $PI_\alpha^{(L1O)}$ is adaptive in the sense that its asymptotic length automatically adjusts according to the asymptotic conditional mean squared prediction error of the underlying point predictor.
\end{remark}

\section{Assumptions on the estimator sequence}
\label{sec:estimators}

\subsection{Discussion of the assumptions}

In this section we discuss the conditions~\ref{c.gauss} and \ref{c.general} used in the proof of Theorem~\ref{thm:L1OPI}. First, \ref{c.gauss}.(\ref{c.gauss1}) states that the ordering of the pairs $(y_i,x_i)$ should have no impact on the estimation of $\beta$. This invariance property should be satisfied by any reasonable estimator since the observations are i.i.d., and thus, exchangeable. Therefore, this assumption is very mild and its sole purpose is to ensure that also the leave-one-out residuals $\tilde{u}_1,\dots, \tilde{u}_n$ are exchangeable random variables.

The conditions~\ref{c.gauss}.(\ref{c.gauss2},\ref{c.gauss3}) are the core ingredients for Theorem~\ref{thm:L1OPI}. They are inspired by the work of \citet{ElKaroui13b} \citep[see also][]{Bean13, ElKaroui13}. 
More specifically, Condition~\ref{c.gauss}.(\ref{c.gauss2}) requires the scaled conditional mean squared prediction error $\E[(x_0'\beta-x_0'\hat{\beta}_n)^2/\sigma^2|Y,X] = \|\Sigma^{1/2}(\hat{\beta}_n-\beta)/\sigma\|_2^2$ to have a non-random limit $\tau^2$. If $p_n/n\to0$, reasonable estimators should be consistent, or even uniformly consistent, which entails that \ref{c.gauss}.(\ref{c.gauss2}) holds with $\tau=0$ and consequently also \ref{c.gauss}.(\ref{c.gauss3}) follows. However, if $p_n/n\to\kappa\in(0,1)$, consistency generally fails. \citet{ElKaroui13b} have argued that in this case $\|\Sigma^{1/2}(\hat{\beta}_n-\beta)/\sigma\|_2$ often has a strictly positive but non-random limit $\tau$ which then provides a meaningful measure for estimation or prediction performance of an estimator sequence $\hat{\beta}_n$. The dependence of $\tau$ on the underlying data generating process and the estimation procedure (e.g., the function $\rho$ in $M$-estimation) is highly non-trivial. First results in that direction are obtained in \citet{ElKaroui13b} and \citet{Bean13} with the goal of choosing an optimal loss function $\rho$. The beauty of Theorem~\ref{thm:L1OPI} and Condition~\ref{c.gauss}.(\ref{c.gauss2}) in our context is that we do not need any information about $\tau$ other than that it is finite. On the contrary, Theorem~\ref{thm:L1OPI}\ref{thm:adaptation} actually shows that the lengths of leave-one-out prediction intervals computed for competing estimators can be used to get an idea about their relative mean squared prediction errors (at least under sufficiently regular error distributions, cf. Remark~\ref{rem:adaptation}). On a technical level, Condition~\ref{c.gauss}.(\ref{c.gauss2}) allows us to derive the asymptotic conditional distribution of the prediction error $y_0 - x_0'\hat{\beta}_n$, given $Y, X$.
Condition~\ref{c.gauss}.(\ref{c.gauss3}) requires that the influence of any single pair of observations $(y_i,x_i)$ on the value of the estimate $\hat{\beta}$ is asymptotically negligible. This condition guarantees that the empirical distribution of the leave-one-out residuals $\tilde{u}_i$ estimates the distribution of the prediction error $y_0-x_0'\hat{\beta}_n$ consistently.

Finally, Condition~\ref{c.general} requires uniform consistency in $\ell_{2+\delta}$-norm. It turns out that this is typically the case for many estimators, even if $p_n/n\to \kappa>0$, in which case consistency in $\ell_2$-norm generally fails \citep[cf.][Section 1.2]{ElKaroui15}. Condition~\ref{c.general} is used to show that the conditional distribution of $x_0'(\hat{\beta}_n-\beta)/\sigma$ converges to the distribution of $l_0N\tau$, in an appropriate sense, even if the $v_{ij}$ are non-Gaussian, and where $N\thicksim \mathcal N(0,1)$ is independent of $l_0$. In this argument, \ref{c.general} ensures the validity of Lyapounov's condition in the corresponding central limit theorem (cf. Section~\ref{sec:proofL1OPI} and specifically Lemma~\ref{lemma:UnifWeak}).

\subsection{Examples}

An important class of estimators that fit in our framework is discussed in \citet{ElKaroui13b}, \citet{Bean13} and \citet{ElKaroui13}. 
In these references the authors argue that many robust $M$-estimators of the form $\hat{\beta}_\rho = \argmin{b\in\R^p}\sum_{i=1}^n \rho(y_i-b'x_i)$ satisfy \ref{c.gauss}.(\ref{c.gauss2},\ref{c.gauss3}), under regularity conditions on the convex loss function $\rho$ and the data generating process, provided that $p_n/n\to \kappa\in(0,1)$. Since these $M$-estimators clearly satisfy also \ref{c.gauss}.(\ref{c.gauss1}), the setting of \citet{ElKaroui13b} constitutes one leading example where our prediction intervals provide valid predictive inference in high dimensions. We also note that many estimators in this class obey the relation $\Sigma^{1/2}(\hat{\beta}_\rho-\beta)/\sigma = \argmin{b\in\R^p}\sum_{i=1}^n\rho(u_i-b'\Sigma^{-1/2}x_i)$, whose distribution does not depend on the parameters $\beta$, $\sigma$ and $\Sigma$.

Another important example of a class of estimators where \ref{c.gauss} can be expected to hold even in a high-dimensional setting is given by the James-Stein type estimators $\hat{\beta}_{JS}(c) = (1-cp/\hat{\beta}_{LS}'X'X\hat{\beta}_{LS})\hat{\beta}_{LS}$, where $\hat{\beta}_{LS}$ is the ordinary least squares estimator. These estimators were examined in a closely related high dimensional predictive inference setting by \citet{Huber13} and \citet{Huber13b}. However, in the case of James-Stein type estimators, the constant $\tau$ in \ref{c.gauss}.(\ref{c.gauss2}) has to be allowed to depend on $n$ and on the parameters $\beta$, $\Sigma$ and $\sigma^2$ which calls for a minor modification of the proof of Theorem~\ref{thm:L1OPI}.

It is also worthwhile to put a special focus on the classical least squares estimator $\hat{\beta}_{LS} = (X'X)^\dagger X'Y$, which is a member of both classes mentioned above. Here, the superscript `$\dagger$' denotes the Moore-Penrose inverse of a matrix. The specifically simple structure of $\hat{\beta}_{LS}$ allows us to verify \ref{c.gauss} and \ref{c.general} under very mild assumptions on the data generating process and to obtain a concrete expression for $\tau$ (see Proposition~\ref{prop:OLS} below, whose proof is deferred to Section~\ref{sec:OLS} in the appendix). Moreover, we point out that in the case of ordinary least squares with regular design ($\det(X'X)\ne 0$), the leave-one-out residuals $\tilde{u}_i = y_i -x_i'\hat{\beta}_{(i),LS}$ obey the relation $\tilde{u}_i = \hat{u}_i/(1-h_i)$, where $h_i = x_i'(X'X)^{-1}x_i$ and $\hat{u}_i = y_i - x_i'\hat{\beta}_{LS}$ is the $i$-th OLS residual. This relation can be used, for example, to compute $PI_\alpha^{(L1O)}$ more efficiently.

\begin{proposition}\label{prop:OLS}
If, in addition to the assumptions of Section~\ref{sec:main}, also $\E[u_0]=0$, $\E[u_0^2] = 1$, $\E[u_0^4] <\infty$ and $p_n/n\to \kappa\in[0,1)$, then the ordinary least squares estimator $\hat{\beta}_{LS} = (X'X)^{\dagger}X'Y$ satisfies the conditions \ref{c.gauss} and \ref{c.general}. Moreover, in this case, the constant $\tau = \tau(\kappa)$ in Condition~\ref{c.gauss}.(\ref{c.gauss2}) depends only on $\kappa$ and on the distribution of $l_0$, and satisfies $\tau(\kappa) = 0$ if, and only if, $\kappa=0$. If, in addition, $l_0^2 = 1$, almost surely, then $\tau(\kappa) = \sqrt{\kappa/(1-\kappa)}$.
\end{proposition}

Finally, we point out that if Conditions~\ref{c.gauss} and \ref{c.general} hold over smaller parameter spaces $\Theta_n \subseteq \R^{p_n}\times (0,\infty)\times S_{p_n}$, then $PI_\alpha^{(L1O)}$ can be shown to have the property \eqref{eq:honest} over this sequence of smaller parameter spaces $\Theta_n$. This is of particular interest, for example, in the case where $p_n> n$ and penalized estimators like the LASSO have to be used. These estimators are typically uniformly consistent over appropriately sparse parameter spaces \citep[cf.][Section 2.4.2]{Buehlmann11}. A qualitatively different parameter space is considered, e.g., in \citet{Lopes15} who shows uniform consistency of ridge regularized estimators under a boundedness assumption on $\|\beta\|_2$ and a specific decay rate of eigenvalues of $\Sigma$. Clearly, these regularized estimators can not satisfy \ref{c.gauss} over the full unrestricted parameter space $\R^{p_n}\times (0,\infty)\times S_{p_n}$.


\section{Concluding remarks}
\label{sec:conclusion}

In this paper we have suggested a generic construction of prediction intervals in linear regression based on leave-one-out residuals. The proposed method has asymptotic conditional coverage probability equal to the nominal level, uniformly in the unknown parameters, even if $p_n$ is allowed to increase with $n$. The leading case we have in mind is $p_n/n\to \kappa\in(0,1)$, as studied in \citet{ElKaroui13b} in the context of $M$-estimation, but our method applies also in the more classical large sample case where $p_n/n\to 0$ as well as in the very high-dimensional case where $p_n\gg n$, provided that regularized estimators are used and the parameter space is restricted appropriately. Our method applies for a large class of estimators including robust $M$-estimators, James-Stein type shrinkage estimators and regularized estimators like LASSO and ridge regression. Moreover, under more restrictive assumptions on the error distribution, the lengths of our intervals adapt to the performance of the underlying estimator used to obtain a linear point prediction in the sense that more accurate predictors (in mean squared prediction error) lead to shorter prediction intervals. This means that our intervals can also be used to evaluate the relative efficiency of competing predictors.

Despite these desirable properties many questions about the leave-one-out prediction interval remain open. First of all, it would be interesting to characterize more explicitly the class of estimators that satisfies Conditions~\ref{c.gauss} and \ref{c.general} needed for our results. Some leading examples have been discussed in Section~\ref{sec:estimators} and further investigations are in progress. 

Moreover, from a practical point of view it is desirable to obtain a rate for the convergence in \eqref{eq:honest}. This would require also a rate in Condition~\ref{c.gauss}.(\ref{c.gauss2}). Moreover, fast estimation of the conditional quantiles of the prediction error would be needed (cf. Section~\ref{sec:quantiles}), which may not be achievable with leave-one-out residuals because of dependence. One idea currently under investigation to resolve this problem is sample splitting. One may simply calculate the estimator $\hat{\beta}_n^{(\nu)}$ only from the first $\lceil \nu n\rceil$ observations of the training sample, where $\nu\in(0,1)$, and use the remaining observations $(y_i,x_i)_{i=\lceil{\nu n\rceil+1}}^n$ to generate residuals $\check{u}_i = y_i - x_i'\hat{\beta}_n^{(\nu)}$ which are conditionally i.i.d. given $\hat{\beta}_n^{(\nu)}$, and which can then be used to estimate the quantiles of the prediction error $y_0-x_0'\hat{\beta}_n^{(\nu)}$. Of course, the price to be paid for this independence is a loss of prediction accuracy, because a smaller sample of size $\lceil \nu n\rceil$ is used to calculate the point prediction.

Finally, the leave-one-out idea for predictive inference in high-dimensional regression problems seems to be applicable far beyond the linear model. But extensions in that direction are left for future research projects.

\section*{Acknowledgements}

The authors thank the participants of the ``ISOR Research Seminar in Statistics and Econometrics" at the University of Vienna for discussion of an early version of the paper. In particular, we want to thank Benedikt P\"otscher and David Preinerstorfer for valuable comments.
The Austrian Science Fund (FWF) supports the first author through project P 28233-N32 and the second author through projects P 28233-N32 and P 26354-N26.

\begin{appendix}

\section{Proofs of main results}

Whenever a quantity that depends on the training sample carries a subscript $(i)$, then it is computed from the training sample where the $i$-th observation pair $(y_i, x_i)$ has been removed. For example, $Y_{(i)} = (y_1,\dots, y_{i-1},y_{i+1},\dots, y_n)'$, $X_{(i)} = [x_1,\dots, x_{i-1}, x_{i+1},\dots, x_n]'$, $\hat{\beta}_{(i)} = \hat{\beta}(Y_{(i)}, X_{(i)})$. Similarly, a double subscript $(ij)$ means that both the $i$-th and the $j$-th observation have been removed. By $\hat{u}_i$ we denote the $i$-th regression residual and by $\tilde{u}_i$ we denote the $i$-th leave-one-out residual, i.e., $\hat{u}_i = y_i - x_i'\hat{\beta}$ and $\tilde{u}_i = y_i - x_i'\hat{\beta}_{(i)}$. The empirical distribution function of the leave-one-out residuals $\tilde{u}_i$ will be denoted by $\tilde{F}_n$. We also write $\tilde{u}_{i(j)} = y_i - x_i'\hat{\beta}_{(ij)}$. Finally, for a cumulative distribution function $F: \R \to [0,1]$, we define the corresponding quantile function by $F^\dagger(t) := \inf\{u\in\R : F(u) \ge t\}$.

\subsection{Proof of Theorem~\ref{thm:L1OPI}}
\label{sec:proofL1OPI}

In order to achieve uniformity in $\beta\in\R^{p_n}$, $\sigma^2>0$ and $\Sigma\in S_{p_n}$ we always consider trending parameter asymptotics, i.e., $\beta = \beta_n\in\R^{p_n}$, $\sigma^2 = \sigma_n^2 >0$ and $\Sigma = \Sigma_n\in S_{p_n}$ are sequences/arrays. We also abbreviate the resulting sequence of induced probability measures $\P_{\beta_n,\sigma_n^2,\Sigma_n}$ on the sample space $\R^{(n+1)\times (p_n +1)}$ (training sample plus one prediction period) simply by $\P_n$.

First, by independence of the training sample $Y, X$ from the realizations in the prediction period $(y_0, x_0)$, we have
\begin{align}
\P_n &\left( y_0 \in PI_\alpha^{(L1O)}(Y,X,x_0) \Big| Y,X \right) \notag
=
\P_n \left( \tilde{q}_{n,\alpha/2} \le y_0 - \hat{\beta}'x_0 \le \tilde{q}_{n,1-\alpha/2} \Big| Y,X\right)\notag\\
&=
\P_n \left( (\beta - \hat{\beta})'x_0 + \sigma u_0 \le \tilde{q}_{n,1-\alpha/2} \Big| Y,X\right)\notag\\
&\quad\quad- 
\P_n \left( (\beta - \hat{\beta})'x_0 + \sigma u_0 < \tilde{q}_{n,\alpha/2} \Big| Y,X\right)\notag\\
&=
F_{b=\Sigma^{1/2}(\beta-\hat{\beta})/\sigma} \left(\tilde{q}_{n,1-\alpha/2}/\sigma\right)
-
F_{b=\Sigma^{1/2}(\beta-\hat{\beta})/\sigma} \left((\tilde{q}_{n,\alpha/2}/\sigma)^-\right), \label{eq:Fb}
\end{align}
where, for $b\in\R^p$, $F_b(t) := \P_n(b'\Sigma^{-1/2}x_0 + u_0 \le t)$ does not depend on $\beta$, $\sigma^2$ and $\Sigma$. Notice that $F_{b=\Sigma^{1/2}(\beta-\hat{\beta})/\sigma}(t) = \P_n( (y_0-x_0'\hat{\beta})/\sigma \le t | Y,X)$ is the distribution function of the scaled conditional prediction error.

The proof now proceeds in two steps.
\begin{enumerate}
\item \label{step1} Show that $\sup_{t\in\R} |F_{b=\Sigma^{1/2}(\beta-\hat{\beta})/\sigma}(t) - F(t)| \to 0$, in probability, where $F(t) = \P_n(l_0 N \tau + u_0 \le t)$ does not depend on $n$, $N\thicksim \mathcal N(0,1)$ is independent of $(l_0,u_0)$ and $\tau\in[0,\infty)$ is as in \ref{c.gauss}.(\ref{c.gauss2}).\footnote{The idea that $(y_0-x_0'\hat{\beta})/\sigma$ should be approximately distributed like $l_0N\tau + u_0$ is inspired by the arguments in \citet{ElKaroui13b} (see page 14562).}
\item \label{step2} Show that the scaled empirical leave-one-out quantiles $\tilde{q}_{n,\alpha}/\sigma$ converge in probability to the corresponding $\alpha$-quantile $q_\alpha$ of $F$.
\end{enumerate}

Note that under either set of assumptions of Theorem~\ref{thm:L1OPI}, $F$ is continuous and strictly increasing. Indeed, if $\tau=0$, the claim is true by assumption. For $\tau>0$ and for $a\in\R$ fixed, $F(a) - F(a^-) = \P_n(l_0 N \tau + u_0 = a) = \E_n[\P_n(N = (a-u_0)/(l_0 \tau) | u_0, l_0)] = 0$ because $|l_0|\ge c$. Moreover, for $a_1< a_2$, $F(a_2) - F(a_1) = \E_n[\P_n(\sign(l_0)N \in [(a_1-u_0)/(|l_0|\tau), (a_2-u_0)/(|l_0|\tau)]|u_0,l_0)] > 0$. Therefore, $q_\alpha = F^{-1}(\alpha)$ is uniquely determined.

The following simple result is somewhat at the heart of the present argument. Its proof is deferred to the appendix. 

\begin{lemma} \label{lemma:UnifWeak}
Fix arbitrary positive constants $\tau\in[0,\infty)$ and $\delta\in(0,2]$ and let $(p_n)_{n\in\N}$ be a sequence of positive integers. Let $u_0$ and $l_0$ be random variables and let $V_0 = (v_{0j})_{j=1}^\infty$ be a sequence of i.i.d. random variables such that $V_0$, $u_0$ and $l_0$ are jointly independent, $|l_0|\ge c>0$, $\E[l_0^2]=1$, $\E[v_{01}]=0$, $\E[v_{01}^2]=1$ and $\E[|v_{01}|^{2+\delta}]<\infty$. For $n\in\N$, define $v_n = (v_{01},\dots, v_{0p_n})'$, $F_{b,n}(t) = \P(l_0b'v_n + u_0\le t)$ and $F(t) = \P(l_0N\tau + u_0\le t)$, where $N\thicksim \mathcal N(0,1)$ is independent of $(l_0,u_0)$. Consider positive sequences $g_1, g_2:\N \to (0,1)$, such that $g_j(n)\to 0$, as $n\to\infty$, $j=1,2$. If $\tau=0$, assume in addition that $t\mapsto \P(u_0\le t)$ is continuous. If $\tau>0$, assume in addition that $p_n\to\infty$ as $n\to\infty$. Then, using the convention that $\sup \emptyset = 0$,
\begin{align} \label{eq:Lemma:UnifWeak}
\sup_{b\in B_n} \sup_{t\in\R} \left| F_{b,n}(t) - F(t) \right| \;\xrightarrow[n\to\infty]{}\; 0,
\end{align}
where the set $B_n$ can be chosen as follows: 
If $\tau=0$, $B_n = B_n^{(0)} = \{b\in\R^{p_n}: \|b\|_2 \le g_1(n)\}$, if $\tau>0$, then $B_n = B_n^{(1)} = \{b\in\R^{p_n} : b \ne 0, | \|b\|_2 - \tau| \le g_1(n), \|b\|_{2+\delta}/\|b\|_2 \le g_2(n) \}$, and if $\tau>0$ and $v_{0j} \thicksim \mathcal N(0,1)$, then $B_n = B_n^{(2)} = \{b\in\R^{p_n} : | \|b\|_2 - \tau| \le g_1(n)\}$.
\end{lemma}


\subsubsection{Convergence of the conditional prediction error distribution} 

Lemma~\ref{lemma:UnifWeak} applies under both parts of Theorem~\ref{thm:L1OPI}, with an appropriate choice of the set $B_n$ corresponding to the three cases $(0)$ $\tau=0$, $(1)$ $\tau>0$ and \ref{c.general} holds, and $(2)$ $\tau>0$ and $v_{ij}\thicksim \mathcal N(0,1)$. This establishes the desired result of step \ref{step1}, i.e., 
\begin{align} \label{eq:UnifWeak}
\sup_{t\in\R}\left| F_{b=\Sigma^{1/2}(\beta-\hat{\beta})/\sigma} (t) 
	- F(t)\right| \;\xrightarrow[n\to\infty]{i.p.}\; 0,
\end{align}
because for $\eps>0$,
\begin{align*}
&\P_n\left( \sup_{t\in\R}\left| F_{b=\Sigma^{1/2}(\beta-\hat{\beta})/\sigma} (t) 
	- F(t)\right| > \eps \right) \\
&\quad\le
\P_n\left( \sup_{t\in\R}\left| F_{b=\Sigma^{1/2}(\beta-\hat{\beta})/\sigma} (t) 
	- F(t)\right| > \eps, \Sigma^{1/2}(\beta-\hat{\beta})/\sigma\in B_n \right) \\
&\hspace{4cm}+
\P_n\left( \Sigma^{1/2}(\beta-\hat{\beta})/\sigma\notin B_n \right)\\
&\quad\le
a_n(\eps) \;+\; \P_n\left( \Sigma^{1/2}(\beta-\hat{\beta})/\sigma\notin B_n \right),
\end{align*}
where $a_n(\eps) = 1$ if $\sup_{b\in B_n} \sup_{t\in\R} \left| F_b(t) - F(t) \right|>\eps$, and $a_n(\eps) = 0$, else.
Choosing the sequences $g_1$ and $g_2$ defining the sets $B_n^{(0)}$, $B_n^{(1)}$ and $B_n^{(2)}$ such that they go to zero sufficiently slowly, we see that, under either set of assumptions of the theorem, the expression on the last line of the previous display converges to zero as $n\to\infty$. Since the same argument applies also when $\hat{\beta}$ is calculated only from a sample of size $n-1$ or $n-2$, \eqref{eq:UnifWeak} also holds for $\hat{\beta}_{(i)}$ or $\hat{\beta}_{(ij)}$ instead of $\hat{\beta}$.

With the result in \eqref{eq:UnifWeak} we can approximate the expression in \eqref{eq:Fb} by
\begin{align}
F(\tilde{q}_{n,1-\alpha/2}/\sigma) - F(\tilde{q}_{n,\alpha/2}/\sigma).
\end{align}

\subsubsection{Consistency of the leave-one-out quantiles}
\label{sec:quantiles}

As an intermediate step, consider the distribution function $F_n(t) := \P_n(\tilde{u}_1/\sigma\le t) = \E_n[\P_n(x_1'(\beta-\hat{\beta}_{(1)})/\sigma + u_1 \le t | \hat{\beta}_{(1)})] = \E_n[ F_{b = \Sigma^{1/2}(\beta-\hat{\beta}_{(1)})/\sigma}(t)]$ and note that
\begin{align}\label{eq:FntoF}
\sup_{t\in\R} | F_n(t) -F(t) | \le \E_n\left[ \sup_{t\in\R} \left| F_{b = \Sigma^{1/2}(\beta-\hat{\beta}_{(1)})/\sigma}(t) - F(t)\right| \right]
\xrightarrow[n\to\infty]{} 0,
\end{align}
because the integrand is bounded and converges to zero by \eqref{eq:UnifWeak} with $\hat{\beta}_{(1)}$ replacing $\hat{\beta}$. 

To deal with the empirical quantiles we use a standard argument. For $\eps>0$, consider
$$
\P_n(|\tilde{q}_{n,\alpha}/\sigma - q_{\alpha}| > \eps)
=
\P_n(\tilde{q}_{n,\alpha}/\sigma > q_{\alpha} + \eps)
+
\P_n(\tilde{q}_{n,\alpha}/\sigma < q_{\alpha} - \eps).
$$
To bound the first probability on the right, abbreviate $J_i = \mathbf 1_{\{\tilde{u}_i/\sigma > q_{\alpha} + \eps\}}$ and note that by virtue of Condition~\ref{c.gauss}.(\ref{c.gauss1}), the $(\tilde{u}_i)_{i=1}^n$, and thus also the $(J_i)_{i=1}^n$ are exchangeable. A basic property of the quantile function \citep[cf.][Lemma 21.1]{vanderVaart07} yields
\begin{align*}
\P(\tilde{q}_{n,\alpha}/\sigma &> q_{\alpha} + \eps)
= \P\left(\alpha > \tilde{F}_n(\sigma(q_{\alpha} + \eps))\right)\\
&= \P\left(1-\tilde{F}_n(\sigma(q_{\alpha} + \eps)) > 1-\alpha \right)\\
&=
\P\left(\frac{1}{n}\sum_{i=1}^n (J_i - \E[J_1]) > 1-\alpha - \E[J_1] \right)\\
&=
\P\left(\frac{1}{n}\sum_{i=1}^n (J_i - \E[J_i]) >  F_n(q_{\alpha} + \eps) - \alpha \right).
\end{align*}
Since $F_n(q_{\alpha} +\eps) \to F(q_\alpha +\eps)>\alpha$, as $n\to\infty$, by \eqref{eq:FntoF}, the probability in the preceding display can be bounded, at least for $n$ sufficiently large, using Markov's inequality, by
\begin{align*}
&(F_n(q_{\alpha} + \eps) - \alpha)^{-2} \E\left[\left| \frac{1}{n}\sum_{i=1}^n (J_i - \E[J_i])\right|^2 \right]\\
&\quad=
(F_n(q_{\alpha} + \eps) - \alpha)^{-2} \left( 
\frac{1}{n} \Var[J_1] + \frac{n(n-1)}{n^2} \Cov(J_1,J_2)
\right),
\end{align*}
where the equality holds in view of the exchangeability of $J_i$. An analogous argument yields a similar upper bound for the probability  $\P(\tilde{q}_{n,\alpha}/\sigma \le q_{\alpha} - \eps)$ but with $(F_n(q_{\alpha} + \eps) - \alpha)^{-2}$ replaced by $(\alpha - F_n(q_{\alpha} - \eps))^{-2}$, and $J_i$ replaced by $K_i = \mathbf 1_{\{\tilde{u}_i/\sigma \le q_{\alpha} - \eps\}}$. The proof will thus be finished if we can show that $\Cov(J_1,J_2)$ and $\Cov(K_1,K_2)$ converge to zero as $n\to\infty$.

We only consider $\Cov(J_1,J_2) = \Cov(1-J_1,1-J_2)$, as the argument for $\Cov(K_1,K_2)$ is analogous. Write $\delta = q_\alpha + \eps$ and $\Cov(1-J_1,1-J_2) = \P(\tilde{u}_1/\sigma\le \delta, \tilde{u}_2/\sigma \le \delta) - \P(\tilde{u}_1/\sigma\le \delta)\P(\tilde{u}_2/\sigma \le \delta)$. Now,
\begin{align*}
\begin{pmatrix}
\tilde{u}_1/\sigma \\
\tilde{u}_2/\sigma
\end{pmatrix}
=
\begin{pmatrix}
\tilde{u}_{1(2)}/\sigma \\
\tilde{u}_{2(1)}/\sigma
\end{pmatrix}
+
\begin{pmatrix}
\tilde{e}_1 \\
\tilde{e}_2
\end{pmatrix},
\end{align*}
where $\tilde{u}_{i(j)} = y_i - x_i'\hat{\beta}_{(ij)}$ and $\tilde{e}_i = (\tilde{u}_i - \tilde{u}_{i(j)})/\sigma = x_i'(\hat{\beta}_{(ij)} - \hat{\beta}_{(i)})/\sigma$, for $i,j\in\{1,2\}$, $i\ne j$. Then $\E[\tilde{e}_i|Y_{(i)}, X_{(i)}] = 0$, and $\E[\tilde{e}_i^2 | Y_{(i)},X_{(i)}] = \|\Sigma^{1/2}(\hat{\beta}_{(i)} - \hat{\beta}_{(ij)})/\sigma\|_2^2$, which converges to zero in probability, under \ref{c.gauss}.(\ref{c.gauss3}), for a sample of size $n-1$ instead of $n$. Hence $\tilde{e}_1$ and $\tilde{e}_2$ converge to zero in probability. The joint distribution function of $\tilde{u}_{1(2)}/\sigma$ and $\tilde{u}_{2(1)}/\sigma$ can be written as
\begin{align*}
&\P_n(\tilde{u}_{1(2)}/\sigma\le s, \tilde{u}_{2(1)}/\sigma \le t)\\
&\quad=
\E_n\left[ \P_n\left(x_1'(\beta-\hat{\beta}_{(12)})/\sigma + u_1 \le s, x_2'(\beta-\hat{\beta}_{(12)})/\sigma + u_2 \le t \Big| Y_{(12)},X_{(12)}\right)\right]\\
&\quad=
\E_n\left[ F_{b=\Sigma^{1/2}(\beta-\hat{\beta}_{(12)})/\sigma}(s) F_{b=\Sigma^{1/2}(\beta-\hat{\beta}_{(12)})/\sigma}(t) \right],
\end{align*}
which converges to $F(s)F(t)$, because the integrand is bounded and converges by \eqref{eq:UnifWeak} with $\hat{\beta}_{(12)}$ replacing $\hat{\beta}$. 
Consequently, the vector $(\tilde{u}_1/\sigma, \tilde{u}_2/\sigma)$ converges weakly to $(L_1, L_2)$, where $L_1$ and $L_2$ are i.i.d. with cdf $F$. Hence, 
$\Cov(J_1,J_2)\to 0$, as $n\to\infty$, and the proof is finished. \hfill{\qed}


\subsection{Proof of Proposition~\ref{prop:OLS}}
\label{sec:OLS}

Condition~\ref{c.gauss}.(\ref{c.gauss1}) is immediate.
For the remaining conditions notice the identity $\Sigma^{1/2}(\hat{\beta}_n - \beta)/\sigma = (\Sigma^{-1/2}X'X\Sigma^{-1/2})^\dagger\Sigma^{-1/2}X'u$, with $u = (u_1,\dots, u_n)'$, whose distribution does not depend on the parameters $\beta$, $\sigma^2$ and $\Sigma$. Hence, without loss of generality, we assume for the rest of this proof that $\beta=0$, $\sigma^2=1$ and $\Sigma = I_{p_n}$. 

For \ref{c.gauss}.(\ref{c.gauss2}), we have to show that $\|\hat{\beta}_n\|_2 \to \tau \in [0,\infty)$, in probability, for a $\tau = \tau(\kappa)$ as in the proposition. To this end, consider the conditional mean
$$
\E[\|\hat{\beta}_n\|_2^2| X] = \trace (X'X)^\dagger X'X (X'X)^\dagger = \trace (X'X)^\dagger \xrightarrow[]{a.s.} \tau^2,
$$
by Lemma~\ref{lemma:traceConv} and for $\tau$ as desired. From the same lemma we get convergence of the conditional variance
\begin{align*}
\Var[\|\hat{\beta}_n\|_2^2| X] &= \Var[u'X(X'X)^{\dagger2} X'u|X] = \Var[u'Ku|X] \\
&=
2\trace K^2 + (\E[u_1^4] - 3)\sum_{i=1}^n K_{ii}^2\\
&\le
2\trace K^2 + (\E[u_1^4] + 3)\sum_{i,j=1}^n K_{ij}^2
=
(\E[u_1^4] + 5)\trace K^2 \\
&= 
(\E[u_1^4] + 5)
\trace X(X'X)^{\dagger2}X'X(X'X)^{\dagger2}X' \\
&= (\E[u_1^4] + 5)
\trace (X'X)^{\dagger2} \xrightarrow[]{a.s.} 0. 
\end{align*}

To establish \ref{c.gauss}.(\ref{c.gauss3}), we abbreviate $S_1 = X_{(1)}'X_{(1)}$ and consider first the event $A_n = \{\lmin(S_1)>0\}$. On this event, also $\lmin(X'X) = \lmin(S_1 + x_1x_1')> 0$, and the Sherman-Morrison formula yields
\begin{align*}
\hat{\beta}_n &= (X'X)^{-1}X'Y = (S_1 + x_1x_1')^{-1}(X_{(1)}'u_{(1)} + x_1 u_1)\\
&= 
\left(S_1^{-1} - \frac{S_1^{-1}x_1x_1'S_1^{-1}}{1 + x_1'S_1^{-1}x_1} \right) (X_{(1)}'u_{(1)} + x_1 u_1) \\
&=
\hat{\beta}_{(1),n} - \frac{S_1^{-1}x_1x_1'\hat{\beta}_{(1),n}}{1 + x_1'S_1^{-1}x_1} + S_1^{-1}x_1u_1 - S_1^{-1}x_1u_1\frac{x_1'S_1^{-1}x_1}{1 + x_1'S_1^{-1}x_1}\\
&=
\hat{\beta}_{(1),n} + \frac{S_1^{-1}x_1(u_1-x_1'\hat{\beta}_{(1),n})}{1 + x_1'S_1^{-1}x_1},
\end{align*}
and thus, $\|\hat{\beta}_n - \hat{\beta}_{(1),n}\|_2^2 = (1+x_1'S_1^{-1}x_1)^{-2} x_1'S_1^{-2}x_1 (u_1 - x_1'\hat{\beta}_{(1),n})^2 \le 2 (u_1^2 + (x_1'\hat{\beta}_{(1),n})^2) x_1'S_1^{\dagger2}x_1$. Clearly, the squared error term $u_1^2$ is bounded in probability because $\E[u_1^2] = 1$; $\E[(x_1'\hat{\beta}_{(1),n})^2|\hat{\beta}_{(1),n}] = \|\hat{\beta}_{(1),n}\|_2^2 \to \tau$, in probability, as above, which implies that $(x_1'\hat{\beta}_{(1),n})^2 = O_\P(1)$; and $\E[x_1'S_1^{\dagger2}x_1|S_1] = \trace S_1^{\dagger2} \to 0$, in probability, by Lemma~\ref{lemma:traceConv}. Therefore, we have $\P(\|\hat{\beta}_n - \hat{\beta}_{(1),n}\|_2^2>\eps, A_n) \le \P(2O_\P(1)o_\P(1)>\eps, A_n) \to 0$. Moreover, $\P(A_n^c) = \P(\lmin(S_1)=0) \to 0$, in view of Lemma~\ref{lemma:traceConv}.

Finally, for \ref{c.general} it suffices to show that $\|\hat{\beta}_n\|_4^4 \to 0$, in probability. Notice that for $M = (m_1, \dots, m_{p_n})' = (X'X)^\dagger X'$, we have
\begin{align*}
\|\hat{\beta}_n\|_4^4 = \|Mu\|_4^4 = \sum_{j=1}^{p_n} (m_j'u)^4
=
\sum_{j=1}^{p_n} \sum_{i_1,i_2,i_3,i_4=1}^n m_{ji_1}m_{ji_2}m_{ji_3}m_{ji_4}u_{i_1}u_{i_2}u_{i_3}u_{i_4}.
\end{align*}
After taking conditional expectation given $X$, only terms with paired indices remain and we get
\begin{align*}
\E[\|\hat{\beta}_n\|_4^4|X] &= \sum_{j=1}^{p_n}\left( \E[u_1^4] \sum_{i=1}^n m_{ji}^4 + 3 \sum_{i\ne k}^n m_{ji}^2m_{jk}^2\right)\\
&\le
\sum_{j=1}^{p_n}\left( \E[u_1^4] \sum_{i,k=1}^n m_{ji}^2m_{jk}^2 + 3 \sum_{i, k=1}^n m_{ji}^2m_{jk}^2\right)\\
&=
(\E[u_1^4] + 3) \sum_{j=1}^{p_n} \|m_j\|_2^4
\le
(\E[u_1^4] + 3) \trace \sum_{i,j=1}^{p_n} m_im_i'm_jm_j'\\
&=
(\E[u_1^4] + 3) \trace (M'M)^2 
= 
(\E[u_1^4] + 3) \trace (X'X)^{\dagger2} \xrightarrow[]{i.p.} 0,
\end{align*}
by Lemma~\ref{lemma:traceConv}.\hfill{\qed}


\subsection{Auxiliary results}
\label{sec:aux}

\begin{proof}[Proof of Lemma~\ref{lemma:UnifWeak}]
First, if $\tau=0$, for every $n\in\N$, take $b_n \in B_n = B_n^{(0)}$ and simply note that $l_0b_n'v_n \to 0$, in probability, and thus $F_{b_n,n}(t) \to F(t)$. Since the limit is continuous, Polya's theorem yields uniform convergence in $t\in\R$. Since $b_n\in B_n$ was arbitrary, we also get uniform convergence over $B_n$.

Next, we consider the gaussian case with $\tau>0$ and take $B_n = B_n^{(2)} = \{b\in\R^p : | \|b\|_2 - \tau| \le g_1(n)\}$. For every $n\in\N$, choose $b_n\in B_n$ arbitrary, and note that $F_{b_n,n}$ is the distribution function of $l_0 b_n'v_n + u_0$, where $v_n \thicksim \mathcal N(0,I_{p_n})$, and $l_0, v_n, u_0$ are independent. Clearly, $l_0 b_n'v_n + u_0 \thicksim l_0 N \|b_n\| + u_0 \to l_0 N \tau + u_0$, weakly, and this limit has continuous distribution function $F$. Hence, by Polya's theorem, $\sup_t |\P(l_0 b_n'v_n + u_0 \le t) - F(t)| \to 0$, as $n\to \infty$. And since the sequence $b_n\in B_n$ was arbitrary, the result follows.

In the general case, let $B_n = B_n^{(1)} = \{b\in\R^p : b \ne 0, | \|b\|_2 - \tau| \le g_1(n), \|b\|_{2+\delta}/\|b\|_2 \le g_2(n) \}$ and first note that $B_n$ may be empty. If $B_n$ is empty only for finitely many indices $n$, then it suffices to consider the non-empty case. If $B_n$ is only finitely many times non-empty, then the desired convergences follows trivially from our convention that $\sup \emptyset = 0$. If $B_n$ is infinitely many times empty and also infinitely many times non-empty, then we restrict to the infinite subsequence $n'$ such that $B_{n'}\ne \emptyset$. It suffices to show that the convergence in \eqref{eq:Lemma:UnifWeak} holds along $n'$. For convenience, we write $n=n'$.
So let $b_n\in B_n$ and define the triangular array $z_{nj} := b_{nj} v_{0j}$, $j=1,\dots, p_n$, which satisfies $\E[z_{nj}]=0$ and $s_n^2 := \sum_{j=1}^p \E[z_{nj}^2] = \|b_n\|_2^2 \ne 0$. The Lyapounov condition is verified by
\begin{align*}
\sum_{j=1}^{p_n} s_n^{-(2+\delta)} \E[|z_{nj}|^{2+\delta}] 
\;&=\;
\E\left[|v_{01}|^{2+\delta}\right]\left(\frac{\|b_n\|_{2+\delta}}{\|b_n\|_2}\right)^{2+\delta}\\
\;&\le\;
\E\left[|v_{01}|^{2+\delta}\right] \left[ g_2(n)\right]^{2+\delta} 
\; \xrightarrow[n\to\infty]{} \; 0.
\end{align*}
Therefore, by Lyapounov's CLT \citep[][Theorem~27.3]{Billingsley95}, we have
$$
b_n'v_n/\|b_n\|_2 \;=\; \sum_{j=1}^{p_n} z_{nj}/s_n \;\xrightarrow[n\to\infty]{w}\; \mathcal N(0,1).
$$
Since $b_n\in B_n$, we must have $\|b_n\|_2 \to \tau$ as $n\to\infty$, and thus, $b_n'v_n = \|b_n\|_2 b_n'v_n/\|b_n\|_2 \xrightarrow[]{w} N\tau$, where $N \thicksim \mathcal N(0,1)$, as $n\to\infty$, and, by independence, $l_0 b_n'v_n + u_0 \xrightarrow[]{w} l_0 N \tau + u_0$. Since the distribution function of this limit is continuous, Polya's theorem yields $\sup_t | F_{b_n,n}(t) - F(t)| \to 0$, as $n\to\infty$. Now the proof is finished because this convergence holds for arbitrary sequences $b_n\in B_n$.
\end{proof}

\begin{lemma}
\label{lemma:traceConv}
Let $\{l_i : i=1,2,\dots\}$ be a sequence of i.i.d. random variables satisfying $|l_1|\ge c>0$, and let $\{v_{ij} : i,j=1,2,\dots\}$ be a double infinite array of i.i.d. random variables with mean zero, unit variance and $\E[v_{11}^4]<\infty$. For positive integers $p\le n$, consider the $n\times p$ random matrix $X = \Lambda V$, where $\Lambda = \diag(l_1,\dots, l_n)$ is diagonal and $V = \{v_{ij} : i=1,\dots, n; j=1,\dots, p\}$. Let $(X'X)^\dagger$ denote the Moore-Penrose pseudo inverse of $X'X$. If $p/n\to \kappa\in[0,1)$ then the following holds:
\begin{enumerate}
	\renewcommand{\theenumi}{(\roman{enumi})}
	\renewcommand{\labelenumi}{{\theenumi}} 
	
	\item\label{lemma:traceConvA} $\liminf \lmin (X'X/n) \ge c^2(1-\sqrt{\kappa})^2$, almost surely.
	\item\label{lemma:traceConvB} If $m>1$, then $\trace(X'X)^{\dagger m} \to 0$, almost surely.
	
	\item \label{lemma:traceConvC} $\trace(X'X)^{\dagger} \to \tau^2$ almost surely, for some constant $\tau=\tau(\kappa)\in [0,\infty)$ that depends only on $\kappa$ and on the distribution of $l_1^2$ and satisfies $\tau(\kappa)=0$ if, and only if, $\kappa=0$.
\end{enumerate}
\end{lemma}

\begin{proof}
Let $\lambda_1\le \dots \le \lambda_p$ and $\mu_1\le \dots \le \mu_p$ denote the ordered eigenvalues of $X'X/n$ and $V'V/n$, respectively, and write $e_i\in\R^n$ for the $i$-th element of the canonical basis in $\R^n$. Then, 
\begin{align*}
\lambda_1 &= \inf_{\|w\|=1} w'V'\Lambda^2Vw/n 
= \inf_{\|w\|=1} \sum_{i=1}^n l_i^2 (e_i'Vw)^2/n \\
&\ge \left(\min_{i=1,\dots, n} l_i^2\right) \inf_{\|w\|=1} w'V'Vw/n 
= c^{2} \mu_1,
\end{align*}
and from the Bai-Yin Theorem \citep{Bai93} it follows that $\mu_1 \to (1-\sqrt{\kappa})^{2}>0$, almost surely, as $p/n\to \kappa \in[0,1)$ \citep[cf.][for the case $\kappa=0$]{Huber13}. Set $\alpha_m := c^{2m} (1-\sqrt{\kappa})^{2m}$ and, for $\alpha>0$, define the functions $h_0$ and $h_\alpha$ by $h_0(y) = 1/|y|$ if $y\ne 0$ and $h(0)=0$, and by $h_\alpha(y) = 1/|y|$, if $|y|> \alpha/2$ and $h_\alpha(y) = 2/\alpha$, if $|y|\le\alpha/2$. With this notation, and from the previous considerations, we see that the difference between
$$
\trace(X'X)^{\dagger m} = n^{-m} \trace(X'X/n)^{\dagger m} 
=  \frac{p}{n^m} \frac{1}{p}\sum_{j=1}^p h_0(\lambda_j^m),
$$
and 
$$
\frac{p}{n^m} \frac{1}{p}\sum_{j=1}^p h_{\alpha_m}(\lambda_j^m)
$$
converges to zero, almost surely, because $\lambda_j^m\ge \lambda_1^m \ge c^{2m}\mu_1^{m} \to \alpha_m > \alpha_m/2>0$, almost surely. But we have $n^{-m} \sum_{j=1}^p h_{\alpha_m}(\lambda_j^m) \le (p/n^m) (2/\alpha_m) \to 0$, if $m>1$ or $\kappa=0$. This finishes part~\ref{lemma:traceConvA}, \ref{lemma:traceConvB} and the case $\kappa=0$ of part~\ref{lemma:traceConvC}.

For the remainder of part~\ref{lemma:traceConvC}, let $m=1$ and $\kappa>0$, and first note that the empirical spectral distribution function $F_n^{\Lambda^2}$ of $\Lambda^2$ is simply given by the empirical distribution function of $l_1^2,\dots, l_n^2$, and this converges weakly (even uniformly) to the distribution function of $l_1^2$, almost surely. Hence, from Theorem~4.3 in \citet{BaiSilv10}, it follows that, almost surely, the empirical spectral distribution function $F_n^{X'X/n}$ of $X'X/n$ converges vaguely, as $p/n\to\kappa\in(0,1)$, to a non-random distribution function $F$ that depends only on $\kappa$ and on the distribution of $l_1^2$. From the argument in the previous paragraph we know that $\lambda_1 \ge c^2\mu_1 \to c^2(1-\sqrt{\kappa})^2=\alpha_1>0$, almost surely, and thus the support of $F$ must be lower bounded by $\alpha_1$. Since $h_{\alpha_1}$ is continuous and vanishes at infinity, by vague convergence, we have \citep[cf.][relation (28.2)]{Billingsley95}
\begin{align*}
\frac{1}{p}\sum_{j=1}^p h_{\alpha_1}(\lambda_j) = \int\limits_{-\infty}^\infty &h_{\alpha_1}(y) dF_n^{X'X/n}(y) \\
&\xrightarrow[]{a.s.} \int\limits_{-\infty}^\infty h_{\alpha_1}(y) dF(y) 
= \int\limits_{-\infty}^\infty \frac{1}{y}\,dF(y) =: \tau_0^2 \in (0, 1/\alpha_1).
\end{align*}
Thus 
$$
\frac{p}{n} \frac{1}{p} \sum_{j=1}^p h_{\alpha_1}(\lambda_j) \quad\xrightarrow[]{a.s.}\quad \kappa \tau_0^2 \;=:\; \tau^2 >0.
$$
\end{proof}


\begin{remark}\normalfont
If the $l_i$ in Lemma~\ref{lemma:traceConv} satisfy $|l_i|=1$, almost surely, then $\tau$ in part~\ref{lemma:traceConvC} is given by $\tau(\kappa)= \kappa/(1-\kappa)$ \citep[cf.][Lemma~B.2]{Huber13}.
\end{remark}

\end{appendix}
{\small
\bibliographystyle{chicago}
\bibliography{../../bibtex/lit}{}
}

\end{document}